\documentclass[
	11pt, 
	abstract,
	DIV=12,
]{scrartcl}
\usepackage{amsmath}
\usepackage{amssymb}
\usepackage{amsthm}
\usepackage{physics}
\usepackage{units}
\usepackage[final]{microtype}
\usepackage{booktabs}
\usepackage{multirow}
\usepackage{xcolor}
\usepackage{tikz}
\usepackage{pgfplots}
\pgfplotsset{compat=newest}
\usetikzlibrary{arrows.meta, external, calc}
\usepgfplotslibrary{colormaps}
\usepackage[backend=biber,hyperref,backref,bibencoding=utf8,defernumbers=true,style=numeric-comp,giveninits=true,maxnames=99,doi=false,isbn=false]{biblatex}
\renewbibmacro{in:}{\ifentrytype{article}{}{\printtext{\bibstring{in}\intitlepunct}}}
\usepackage{csquotes}
\makeatletter
\let\blx@rerun@biber\relax
\makeatother

\usepackage{hyperref}
\usepackage{cleveref}

\theoremstyle{plain}
\newtheorem{theorem}{Theorem}
\newtheorem{lemma}{Lemma}

\newcommand{\diam}{{\operatorname{diam}\,}}

\newcommand{\R}{{\mathbb R}}
\newcommand{\N}{{\mathbb N}}
\renewcommand{\vec}[1]{\boldsymbol{#1}}
\renewcommand{\tilde}[1]{\widetilde{#1}}
\renewcommand{\hat}[1]{\widehat{#1}}
\newcommand{\BDMIk}[1]{{I_{#1}^{\operatorname{BDM}}}}
\newcommand{\hBDMIk}[1]{{\hat{I}_{#1}^{\operatorname{BDM}}}}
\newcommand{\tBDMIk}[1]{{\tilde{I}_{#1}^{\operatorname{BDM}}}}
\renewcommand{\div}{{\operatorname{div}\,}}
\newcommand{\hdiv}{{\widehat{\operatorname{div}}\,}}
\newcommand{\tdiv}{{\widetilde{\operatorname{div}}\,}}
\newcommand{\Poly}[3][]{P^{#1}_{#2}(#3)}
\newcommand{\Polyvec}[3][]{\vec{P}^{#1}_{#2}(#3)}
\newcommand{\Hdiv}{\vec{H}_{\mathrm{div}}}

\newcommand{\abssec}[1]{\noindent\small {\bfseries #1\quad}\ignorespaces}
\renewenvironment{abstract}{\abssec{Abstract}}{\par\vspace{1em}}
\newenvironment{keywords}{\abssec{Keywords}}{\par\vspace{1em}}
\newenvironment{MSC}{\abssec{Mathematics Subject Classification (2020)}}{\par\vspace{1em}}

\title{Brezzi--Douglas--Marini interpolation on anisotropic simplices and prisms}
\date{\today}
\author{Volker Kempf}

\begin{document}
\maketitle

\begin{abstract}
	The Brezzi--Douglas--Marini interpolation error on anisotropic elements has been analyzed in two recent publications, the first focusing on simplices with estimates in $L^2$, the other considering parallelotopes with estimates in terms of $L^p$-norms.
	This contribution provides generalized estimates for anisotropic simplices for the $L^p$ case, $1\leq p\leq\infty$, and shows new estimates for anisotropic prisms with triangular base.
\end{abstract}
\begin{MSC} 
	65D05, 65N30
\end{MSC}

\section{Introduction}
The Brezzi--Douglas--Marini (BDM) finite element \cite{Nedelec1986} was introduced to approximate $\Hdiv$ by polynomials.
This proves useful for problems of incompressible fluid flow, where recent approaches employ $\Hdiv$-conforming discretizations to approximate the velocity solution.
The corresponding interpolation operator can also be used as reconstruction operator to gain pressure-robust methods in the spirit of \cite{Linke2014}.
Boundary layers or edge singularities in these problems require the use of anisotropic, i.e., highly stretched elements, so interpolation error estimates for such settings are required, see \cite{ApelKempf2021}.
In \cite{ApelKempf2020}, these estimates were shown for simplicial elements in terms of the $\vec{L}^2$-norm, and \cite{Franz2021} contains estimates for anisotropic parallelotopes in terms of $\vec{L}^p$-norms.

The main focus in \cite{Franz2021} is on estimates in the $\Hdiv$-norm and while the technique for these proofs can also be used for simplices, it is not applicable for the case of prisms as defined in \cite{Nedelec1986}, as here the commuting diagram property \cite[(3)]{Franz2021} is not satisfied.
Interestingly, it is neither satisfied on cubes as defined in \cite{Nedelec1986}; here \cite{Franz2021} uses the original definition of the BDM elements, see, e.g., \cite{BoffiBrezziFortin2013}.

In the following sections we provide a generalization of the results from \cite{ApelKempf2020} to $\vec{L}^p$ spaces, $1\leq p \leq \infty$, and show new interpolation error estimates for anisotropic triangular prisms. The results with detailed proofs are contained in the author's PhD thesis \cite{Kempf2022:PhD} that is to appear.

\textsc{Notation:} Vectors, vector valued functions and the spaces of such functions are set in bold and we use $I_n = \{1,\ldots,n\}$ for index sets.
The spatial dimension is denoted by $d$ and the expression $a \lesssim b$ means there is a positive constant $C$ so that $a \leq C b$.
The norms of the Sobolev spaces $W^{m,p}(G)$ are denoted by $\norm{\cdot}_{m,p,G}$, multi-indices by $\vec{\alpha}$.
The directional derivative in direction $\vec{l}$ is denoted by $\pdv{}{\vec{l}}$.

\section{Error estimates on anisotropic simplices}
We state the anisotropic interpolation error estimates on simplices in terms of $\vec{L}^p$-norms, which is a generalization of the results from \cite{ApelKempf2020}, where only $p=2$ was considered.
As the proofs are largely analogous to those in \cite{ApelKempf2020}, we omit them here.
The principal idea of the proof is to first show stability estimates on the reference elements and then transfer them in two steps first to an element of a reference family using an affine transformation with a diagonal matrix and then to the general element.

Recall that a simplex is said to satisfy the \emph{maximum angle condition} if all angles within and between facets are bounded by a constant $\bar{\phi}<\pi$.
In addition, it satisfies the \emph{regular vertex property} if for one vertex there is a constant $\bar{c}>0$ so that $\abs{\det N} \geq \bar{c}$ holds for the matrix $N$ whose columns consist of the unit vectors $\vec{l}_i$, $i\in I_d$, along the outgoing edges from this vertex.
The lengths of the edges corresponding to the vectors $\vec{l}_i$ are the element size parameters and are denoted by $h_i$. The degrees of freedom for the BDM interpolation operator $\BDMIk{k}$ of order $k$ on simplices can be found in \cite[p. 59]{Nedelec1986}.

The error estimates for the BDM interpolation depending on the geometric regularity are given by the two following theorems, cf. \cite[Theorems 4.3, 4.4]{ApelKempf2020} and \cite[Theorems 6.2, 6.3]{AcostaApelDuranLombardi2011}.
\begin{theorem}
	Let a simplicial element $T$ satisfy the regular vertex property with constant $\bar{c}$. Then for $k\geq 1$, $0\leq m\leq k$ and $\vec{v}\in \vec{W}^{m+1,p}(T)$, $1\leq p\leq\infty$, the estimate
	\begin{equation*}
		\norm{\vec{v}-\BDMIk{k}\vec{v}}_{0,p,T} \lesssim \sum_{\abs{\vec{\alpha}}=m+1} h^{\vec{\alpha}} \norm{D^{\vec{\alpha}}_{\vec{l}}\vec{v}}_{0,p,T} + h_T \sum_{\abs{\vec{\alpha}}=m} h^{\vec{\alpha}} \norm{D^{\vec{\alpha}}_{\vec{l}} \div\vec{v}}_{0,p,T}
	\end{equation*}
	holds, the constant only depends on $\bar{c}$ and $k$. Here $h_T=\diam T$, $h^{\vec{\alpha}} = \prod_{j\in I_d}h_j^{\alpha_j}$ and $D_{\vec{l}}^{\vec{\alpha}} = \frac{\partial^{\abs{\vec{\alpha}}}}{\partial_{\vec{l}_1}^{\alpha_1} \cdots \partial_{\vec{l}_d}^{\alpha_d}}$.
\end{theorem}

\begin{theorem}
	Let a simplicial element $T$ satisfy the maximum angle condition with constant $\bar{\phi}$. Then for $k\geq 1$, $0\leq m \leq k$ and $\vec{v}\in \vec{W}^{m+1,p}(T)$, $1\leq p\leq\infty$, the estimate
	\begin{equation*}
		\norm{\vec{v}-\BDMIk{k}\vec{v}}_{0,p,T} \lesssim h_T^{m+1} \norm{D^{m+1} \vec{v}}_{0,p,T}
	\end{equation*}
	holds, the constant only depends on $\bar{\phi}$ and $k$. The notation $D^n$ means the sum of the absolute values of all derivatives of order $n$.
\end{theorem}

\section{Error estimates on anisotropic prisms}
Similar results can be achieved for triangular prisms with some modifications.
The prismatic reference element $\hat{P}$ with the notation for the vertices is given on the left hand side of \Cref{fig:prism_element}.
\begin{figure}[t]
	\tikzsetnextfilename{Prism_reference_element}
\begin{tikzpicture}[scale=1]
	\begin{axis}[
		axis lines=center,
		view={120}{35},
		xtick={1}, ytick={1}, ztick={1},
		xtick align=outside, ytick align=outside, ztick align=outside,
		xticklabel style={anchor=east}, yticklabel style={anchor=west},
		xlabel=$\hat{x}_1$, ylabel=$\hat{x}_2$,zlabel=$\hat{x}_3$,
		xlabel style={anchor=east}, ylabel style={anchor=south west}, zlabel style={anchor=south east},
		xmin=0, xmax=1.5,
		ymin=0, ymax=1.5,
		zmin=0, zmax=1.5,
		scale only axis,
		width=0.3\textwidth,
		height=0.3\textwidth,
	]
		\draw[thick, blue, rounded corners=0.1pt] (1,0,1) -- (0,0,1) -- (0,1,1) -- (0,1,0) -- (1,0,0) -- (1,0,1) -- (0,1,1);
		\draw[thick, blue, densely dashed] (1,0,0) -- (0,0,0) -- (0,1,0);
		\draw[thick, blue, densely dashed] (0,0,0) -- (0,0,1);
		\node[anchor=north west] at (1,0,0) {$\vec{p}_1$};
		\node[anchor=north east] at (0,1,0) {$\vec{p}_2$};
		\node[anchor=east] at (0,0,0) {$\vec{p}_3$};
		\node[anchor=east] at (1,0,1) {$\vec{p}^1$};
		\node[anchor=north east] at (0,1,1) {$\vec{p}^2$};
		\node[anchor=south west] at (0,0,1) {$\vec{p}^3$};
		\node at (0,1.25,1.25) {$\hat{P}$};
	\end{axis}
	
\end{tikzpicture}%
	\hfill
	\begin{tikzpicture}
	\draw[white] (0,0) -- (1,0);
	\draw[-stealth] (0,2) -- node[above] {$J_{\tilde{P}}$} (2,2);
	\end{tikzpicture}
	\hfill
	\tikzsetnextfilename{Prism_element}
\begin{tikzpicture}[scale=1]
	\begin{axis}[
		axis lines=center,
		view={120}{35},
		xtick={0.2}, ytick={0.1}, ztick={1},
		xticklabels={$h_1$}, yticklabels={$h_2$}, zticklabels={$h_3$},
		xtick align=outside, ytick align=outside, ztick align=outside,
		xticklabel style={anchor=east}, yticklabel style={anchor=west},
		xlabel=$\tilde{x}_1$, ylabel=$\tilde{x}_2$, zlabel=$\tilde{x}_3$,
		xlabel style={anchor=east}, ylabel style={anchor=south west}, zlabel style={anchor=south east},
		xmin=0, xmax=1.5,
		ymin=0, ymax=1.5,
		zmin=0, zmax=1.5,
		scale only axis,
		width=0.3\textwidth,
		height=0.3\textwidth,
	]		
		\draw[thick, blue, rounded corners=0.1pt] (0.2,0,1) -- (0,0,1) -- (0,0.1,1) -- (0,0.1,0) -- (0.2,0,0) -- (0.2,0,1) -- (0,0.1,1);
		\draw[thick, blue, densely dashed] (0.2,0,0) -- (0,0,0) -- (0,0.1,0);
		\draw[thick, blue, densely dashed] (0,0,0) -- (0,0,1);
		\node at (0,1.25,1.25) {$\tilde{P}$};
	\end{axis}
\end{tikzpicture}%
	\caption{Reference prism $\hat{P}$ with vertex numbering, and a transformed prism of the reference family $\mathcal{R}_P$.}\label{fig:prism_element}
\end{figure}
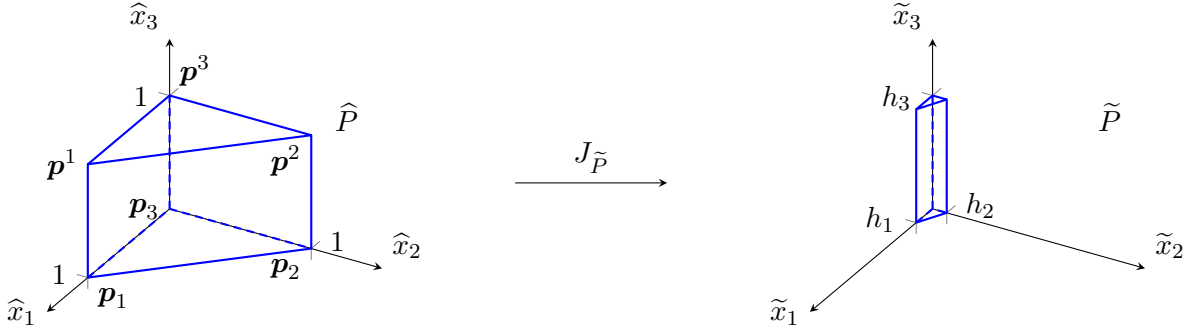
The element $\hat{P}$ is transformed to an element $\tilde{P}$ of the reference family $\mathcal{R}_P$ by the transformation $\tilde{\vec{x}} = J_{\tilde{P}} \hat{\vec{x}}$, where
\begin{equation}\label{eq:trafo}
	J_{\tilde{P}} = \begin{pmatrix}h_1&0&0\\0&h_2&0\\0&0&h_3\end{pmatrix},
\end{equation}
see \Cref{fig:prism_element}.
The vertical facet of $\hat{P}$ opposite of the vertices $\vec{p}_i$ and $\vec{p}^i$ is denoted by $e_i$, the horizontal facet at $\hat{x}_3 = 0$ by $e_b$ and the one at $\hat{x}_3 = 1$ by $e_t$.
The facet normals of $\hat{P}$ are thus given by
\begin{equation*}
	\vec{n}_{\hat{P},e_1} = \begin{pmatrix}-1\\0\\0\end{pmatrix}, \quad \vec{n}_{\hat{P},e_2} = \begin{pmatrix}0\\-1\\0\end{pmatrix}, \quad \vec{n}_{\hat{P},e_3} = \frac{1}{\sqrt{2}}\begin{pmatrix}1\\1\\0\end{pmatrix}, \quad \vec{n}_{\hat{P},e_b} = \begin{pmatrix}0\\0\\-1\end{pmatrix}, \quad \vec{n}_{\hat{P},e_t} = \begin{pmatrix}0\\0\\1\end{pmatrix}.
\end{equation*}

The local BDM interpolation operator $\BDMIk{k}$ of order $k$ on a prism $P$ maps into the space $\Polyvec{k,k}{P}$ of vector valued polynomials of total degree $k$ in $x_1$ and $x_2$ and degree $k$ in $x_3$.
It is defined by the functionals, see \cite[p. 64]{Nedelec1986},
\begin{subequations}\label{eq:interpolation_relations}
	\begin{align}
		\int_{e_i} \BDMIk{k}\vec{v} \cdot \vec{n}_{P,e_i} z \dd{\vec{s}} &= \int_{e_i} \vec{v} \cdot \vec{n}_{P,e_i} z \dd{\vec{s}} &&\forall z\in \Poly{k}{e_i}, \quad i\in \{b,t\}, \\
		\int_{e_i} \BDMIk{k}\vec{v} \cdot \vec{n}_{P,e_i} z \dd{\vec{s}} &= \int_{e_i} \vec{v} \cdot \vec{n}_{P,e_i} z \dd{\vec{s}} &&\forall z\in Q_k(e_i), \quad i\in I_3, \\
		\int_P (\BDMIk{k}\vec{v})_3 z_3 \dd{\vec{x}} &= \int_P v_3 z_3 \dd{\vec{x}} &&\forall z_3\in \Poly{k,k-2}{P}, \\
		\int_P (\BDMIk{k}\vec{v})_1 z_1 + (\BDMIk{k}\vec{v})_2 z_2 \dd{\vec{x}} &= \int_P v_1 z_1 + v_2 z_2 \dd{\vec{x}} &&\forall (z_1,z_2)\in \mathcal{P}_{k-1,k}(P), 
	\end{align}
\end{subequations}
where $P_k(e)$ is the space of polynomials with maximal degree $k$, $Q_k(e)$ is the space of polynomials with degree $k$ in each of the dimensions of the facet $e$, $\Poly{m,n}{P}$ is the space of polynomials with total degree $m$ in $x_1$ and $x_2$, and degree $n$ in $x_3$.
The space $\mathcal{P}_{m,n}(P)$ consists of pairs of polynomials with degree $n$ in $x_3$, which are for fixed $x_3$, i.e., on the triangle $T_{x_3}$, in the space $\vec{N}_{m}(T_{x_3}) = \vec{P}_{m-1}(T_{x_3}) \oplus \vec{S}_{m}(T_{x_3})$, where $\vec{S}_{m}(T_{x_3}) = \{\vec{p}\in \vec{P}_{m}(T_{x_3}):\vec{p}(\vec{x})\cdot \vec{x}=0\ \forall\vec{x}\in T_{x_3}\}$.

We first require the analog of \cite[Lemma 3.1]{ApelKempf2020} for the prism reference element.
The proof is omitted here for brevity since it follows along the same steps as the proof for the analogous lemma for simplices.
\begin{lemma}\label{lem:BDMI_technical_prism}
	Let $\hat{P}$ be the reference element from \Cref{fig:prism_element}, $\hat{f}_j\in L^p(e_j)$, $j\in I_2$, $\hat{f}_3\in L^p(e_b)$, $1\leq p\leq\infty$, and 
	\begin{align*}
		&\hat{\vec{u}}(\hat{\vec{x}}) = \begin{pmatrix} \hat{f}_1(\hat{x}_2,\hat{x}_3),0,0\end{pmatrix}^T,&&\hat{\vec{v}}(\hat{\vec{x}}) = \begin{pmatrix}0, \hat{f}_2(\hat{x}_1,\hat{x}_3),0\end{pmatrix}^T,&&\hat{\vec{w}}(\hat{\vec{x}}) = \begin{pmatrix}0,0,\hat{f}_3(\hat{x}_1,\hat{x}_2)\end{pmatrix}^T.
	\end{align*}
	Then there are functions $\hat{q}_j\in P_k(e_j)$, $j\in I_2$, $\hat{q}_3\in P_k(e_b)$, so that
	\begin{align*}
		&\hBDMIk{k}\hat{\vec{u}} = \begin{pmatrix} \hat{q}_1(\hat{x}_2,\hat{x}_3),0,0\end{pmatrix}^T,&&\hBDMIk{k}\hat{\vec{v}} = \begin{pmatrix}0,\hat{q}_2(\hat{x}_1,\hat{x}_3),0\end{pmatrix}^T,&& \hBDMIk{k}\hat{\vec{w}} = \begin{pmatrix}0,0,\hat{q}_3(\hat{x}_1,\hat{x}_2)\end{pmatrix}^T.
	\end{align*}		
\end{lemma}

With this lemma, we get the stability estimate on the reference element.
\begin{lemma}\label{lem:stability_BDM_reference_element_prism}
	Let $\hat{\vec{u}}\in\vec{W}^{1,p}(\hat{P})$, $1\leq p\leq\infty$. Then we have the estimates
	\begin{subequations}
		\begin{align*}
			\norm{(\hBDMIk{k} \hat{\vec{u}})_i}_{0,p,\hat{P}} &\lesssim \norm{\hat{u}_i}_{1,p,\hat{P}} + \norm{\hdiv\hat{\vec{u}}}_{0,p,\hat{P}} + \norm{\pdv{\hat{u}_3}{\hat{x}_3}}_{0,p,\hat{P}}, \qquad i\in I_2,\\
			\norm{(\hBDMIk{k} \hat{\vec{u}})_3}_{0,p,\hat{P}} &\lesssim \norm{\hat{u}_3}_{1,p,\hat{P}} + \norm{\hdiv\hat{\vec{u}}}_{0,p,\hat{P}}.
		\end{align*}	
	\end{subequations}
\end{lemma}
\begin{proof}
	The proofs for $i\in I_2$ are analogous so we only show the details for the first and third components, starting with $i=1$.
	Let
	\begin{align*}
		&\hat{\vec{u}}_* = \begin{pmatrix}0\\\hat{u}_2(\hat{x}_1,0,\hat{x}_3)\\\hat{u}_3(\hat{x}_1,\hat{x}_2,0)\end{pmatrix}, &&\hat{\vec{v}}=\hat{\vec{u}}-\hat{\vec{u}}_* = \begin{pmatrix}\hat{u}_1\\\hat{u}_2-\hat{u}_2(\hat{x}_1,0,\hat{x}_3)\\\hat{u}_3-\hat{u}_3(\hat{x}_1,\hat{x}_2,0)\end{pmatrix}.
	\end{align*}
	The previous \Cref{lem:BDMI_technical_prism} thus yields $(\hBDMIk{k} \hat{\vec{v}})_1 = (\hBDMIk{k}\hat{\vec{u}})_1$, and it holds $\hdiv\hat{\vec{v}} = \hdiv\hat{\vec{u}} - \hdiv\hat{\vec{u}}_* = \hdiv\hat{\vec{u}}$. 
	We define the two functions
	\begin{align*}
		&\hat{\vec{v}}_* = \begin{pmatrix}0\\\hat{x}_2 \hat{q}_2\\\hat{x}_3\hat{q}_3\end{pmatrix}, &&\hat{\vec{w}}=\hat{\vec{v}}-\hat{\vec{v}}_* = \begin{pmatrix}\hat{v}_1\\\hat{v}_2-\hat{x}_2\hat{q}_2\\\hat{v}_3-\hat{x}_3\hat{q}_3\end{pmatrix},
	\end{align*}
	where $\hat{q}_2\in \Poly{k-1,k}{\hat{P}}$, $\hat{q}_3\in \Poly{k,k-1}{\hat{P}}$, so that
	\begin{align*}
		\int_{\hat{P}} \hat{w}_2 z \dd{\vec{x}}=\int_{\hat{P}} (\hat{v}_2 - \hat{x}_2\hat{q}_2)z \dd{\vec{x}}=0 \qquad \forall z\in \Poly{k-1,k}{\hat{P}}, \\
		\int_{\hat{P}} \hat{w}_3 z \dd{\vec{x}}=\int_{\hat{P}} (\hat{v}_3 - \hat{x}_3\hat{q}_3)z \dd{\vec{x}}=0 \qquad \forall z\in \Poly{k,k-1}{\hat{P}}.
	\end{align*}
	This means that the functions $\hat{q}_2$ and $\hat{q}_3$ are well defined.		
	Since $\hat{\vec{v}}_*\in\Polyvec{k,k}{\hat{P}}$ it follows that $\hBDMIk{k}\hat{\vec{v}}_* = \hat{\vec{v}}_*$ and thus $(\hBDMIk{k} \hat{\vec{w}})_1 = (\hBDMIk{k} \hat{\vec{v}})_1 = (\hBDMIk{k} \hat{\vec{u}})_1$.
	The interpolated function $\hBDMIk{k} \hat{\vec{w}} = \hat{\vec{t}} = (\hat{t}_1,\hat{t}_2,\hat{t}_3)^T$ is then defined by the relations, see \eqref{eq:interpolation_relations},
	\begingroup
	\allowdisplaybreaks
	\begin{align*}
		\int_{e_b} \hat{t}_3 z \dd{\vec{s}} &= \int_{e_b} \hat{w}_3 z \dd{\vec{s}} = 0 && \forall z\in \Poly{k}{e_b}, \\
		\int_{e_t} \hat{t}_3 z \dd{\vec{s}} &= \int_{e_t} \hat{w}_3 z \dd{\vec{s}} = \int_{e_t} \hat{w}_3 z \dd{\vec{s}} - \int_{e_b} \hat{w}_3 z \dd{\vec{s}} = \int_{\hat{P}} \pdv{\hat{w}_3}{\hat{x}_3} z \dd{\vec{x}}  && \forall z\in \Poly{k}{e_t}, \\
		\int_{e_1} \hat{t}_1 z \dd{\vec{s}} &= \int_{e_1} \hat{w}_1 z \dd{\vec{s}} && \forall z \in Q_k(e_1),\\
		\int_{e_2} \hat{t}_2 z \dd{\vec{s}} &= \int_{e_2} \hat{w}_2 z \dd{\vec{s}} = 0 && \forall z \in Q_k(e_2), \\
		\int_{e_3} (\hat{t}_1 + \hat{t}_2) z \dd{\vec{s}} &= \int_{e_3} (\hat{w}_1 + \hat{w}_2) z \dd{\vec{s}}  && \forall z \in Q_k(e_3), \\
		\int_{\hat{P}} \hat{t}_3 z_3 \dd{\vec{x}} &= \int_{\hat{P}} \hat{w}_3 z_3 \dd{\vec{x}} = \int_{\hat{P}} (\hat{v}_3 - \hat{x}_3 \hat{q}_3) z_3 \dd{\vec{x}} = 0 &&\forall z_3 \in \Poly{k,k-2}{\hat{P}}, \\
		\int_{\hat{P}} \hat{t}_1 z_1 + \hat{t}_2 z_2 \dd{\vec{x}} &= \int_{\hat{P}} \hat{w}_1 z_1 + \hat{w}_2 z_2 \dd{\vec{x}} = \int_{\hat{P}} \hat{w}_1 z_1 \dd{\vec{x}} &&\forall (z_1,z_2) \in \mathcal{P}_{k-1,k}(\hat{P}),
	\end{align*}
	\endgroup
	where the definitions of $\hat{q}_2$, $\hat{q}_3$, and that $\hat{w}_2|_{e_2} \equiv 0$, $\hat{w}_3|_{e_b} \equiv 0$ were used.
	Additional computations for the relation on $e_3$ yield
	\begin{align*}
		\frac{1}{\sqrt{2}}\int_{e_3} (\hat{t}_1 + \hat{t}_2) z \dd{\vec{s}} &= \frac{1}{\sqrt{2}}\int_{e_3} (\hat{w}_1 + \hat{w}_2) z \dd{\vec{s}} = \int_{e_3} \hat{\vec{w}} \cdot \vec{n}_{e_3} \dd{\vec{s}} \\
		&= \int_{\hat{P}} \left(\pdv{\hat{w}_1}{\hat{x}_1} + \pdv{\hat{w}_2}{\hat{x}_2}\right) z \dd{\vec{x}} + \int_{\hat{P}} \hat{w}_1 \pdv{z}{\hat{x}_1} \dd{\vec{x}} - \int_{\partial\hat{P}\setminus e_3} \hat{\vec{w}}\cdot \vec{n}_{\partial\hat{P}} z \dd{\vec{s}}\\
		&= \int_{\hat{P}} \left( \hdiv\hat{\vec{w}} -2\pdv{\hat{w}_3}{\hat{x}_3}\right) \dd{\vec{x}} + \int_{\hat{P}} \hat{w}_1 \pdv{z}{\hat{x}_1} \dd{\vec{x}} - \int_{e_1} \hat{w}_1 z\dd{\vec{s}}.
	\end{align*}
	Thus, the terms
	\begin{subequations}
		\begin{align*}
			\int_{e_1} \hat{w}_1 z \dd{\vec{s}} &= \int_{e_1} \hat{u}_1 z \dd{\vec{s}},&
			\int_{\hat{P}} \hat{w}_1 z \dd{\vec{x}} &= \int_{\hat{P}} \hat{u}_1 z \dd{\vec{x}}, \\
			\int_{\hat{P}} \pdv{\hat{w}_3}{\hat{x}_3} z \dd{\vec{x}} &= \int_{\hat{P}} \pdv{\hat{u}_3-\hat{x}_3\hat{q}_3}{\hat{x}_3} z \dd{\vec{x}},&
			\int_{\hat{P}} (\hdiv \hat{\vec{w}}) z \dd{\vec{x}} &= \int_{\hat{P}} \hdiv (\hat{\vec{u}} - \hat{\vec{v}}_*) z\dd{\vec{x}}
		\end{align*}
	\end{subequations}
	define the interpolant. We get the desired estimate
	\begin{equation*}
		\norm{(\hBDMIk{k} \hat{\vec{u}})_1}_{0,p,\hat{P}} \lesssim \norm{\hat{u}_1}_{1,p,\hat{P}} + \norm{\hdiv\hat{\vec{u}}}_{0,p,\hat{P}} + \norm{\pdv{\hat{u}_3}{\hat{x}_3}}_{0,p,\hat{P}}
	\end{equation*}
	using a trace theorem and steps similar to those in the proof of \cite[Lemma 3.3]{AcostaApelDuranLombardi2011} to estimate the terms involving $\hat{x}_3\hat{q}_3$ and $\hdiv \hat{\vec{v}}_*$.	
	For the third component, with the definitions
	\begin{align*}
		&\hat{\vec{v}}=\hat{\vec{u}}-\hat{\vec{u}}_* = \begin{pmatrix}\hat{u}_1-\hat{u}_1(0,\hat{x}_2,\hat{x}_3)\\\hat{u}_2-\hat{u}_2(\hat{x}_1,0,\hat{x}_3)\\\hat{u}_3\end{pmatrix}, && \hat{\vec{v}}_* = \begin{pmatrix} \hat{x}_1\hat{q}_1\\ \hat{x}_2\hat{q}_2\\ 0	\end{pmatrix}, &&\hat{\vec{w}}=\hat{\vec{v}} - \hat{\vec{v}}_* = \begin{pmatrix}\hat{v}_1-\hat{x}_1\hat{q}_1\\\hat{v}_2-\hat{x}_2\hat{q}_2\\\hat{v}_3\end{pmatrix},
	\end{align*}
	where the functions $\hat{q}_j \in \Poly{k-1,k}{\hat{P}}$, $j\in I_2$, are now defined by
	\begin{align*}
		\int_{\hat{P}} \hat{w}_j z \dd{\vec{x}}=\int_{\hat{P}} (\hat{v}_j - \hat{x}_j\hat{q}_j)z \dd{\vec{x}}=0 \qquad \forall z\in \Poly{k-1,k}{\hat{P}},
	\end{align*}
	the relevant terms of the interpolation relations are 
	\begin{align*}
		\int_{e_b} \hat{w}_3 z \dd{\vec{s}} &= \int_{e_b} \hat{u}_3 z \dd{\vec{s}}, &
		\int_{e_t} \hat{w}_3 z \dd{\vec{s}} &= \int_{e_t} \hat{u}_3 z \dd{\vec{s}}, \\
		\int_{\hat{P}} \hat{w}_3 z_3 \dd{\vec{x}} &= \int_{\hat{P}} \hat{u}_3 z_3 \dd{\vec{x}}, &
		\int_{\hat{P}} \hdiv \hat{\vec{w}} z \dd{\vec{x}} &= \int_{\hat{P}} \hdiv (\hat{\vec{u}} - \hat{\vec{v}}_*) z\dd{\vec{x}}.
	\end{align*}
	From here the same steps as for the first component yield the desired estimate.
\end{proof}

Using the transformation \eqref{eq:trafo} we bring the stability estimate to an element of the reference family.
\begin{lemma}
	Let $\tilde{P} = J_{\tilde{P}} \hat{P} + \vec{x}_0$, $\vec{x}_0 \in \R^3$, and $\tilde{\vec{v}}\in\vec{W}^{1,p}(\tilde{P})$, $1\leq p \leq \infty$. Then on the prism $\tilde{P}$ the estimate
	\begin{equation*}
		\norm{\tBDMIk{k}\tilde{\vec{v}}}_{0,p,\tilde{P}} \lesssim \sum_{\abs{\vec{\alpha}}\leq 1} h^{\vec{\alpha}} \norm{D^{\vec{\alpha}}\tilde{\vec{v}}}_{0,p,\tilde{P}} + h_{\tilde{P}} \norm{\tdiv \tilde{\vec{v}}}_{0,p,\tilde{P}} + (h_1+h_2)\norm{\pdv{\tilde{v}_3}{\tilde{x}_3}}_{0,p,\tilde{P}}
	\end{equation*}
	holds, where $h_{\tilde{P}} = \max\{h_1, h_2, h_3\}$.
\end{lemma}
\begin{proof}
	Using \Cref{lem:stability_BDM_reference_element_prism} and the relations
	\begin{align*}
		\norm{\tilde{\vec{w}}}_{0,p,\tilde{P}} &= \bigg(\int_{\tilde{P}} \sum_{i\in I_d} \abs{\tilde{w}_i}^p \dd{\vec{x}}\bigg)^{\nicefrac{1}{p}} \\
		&\leq (\det J_{\tilde{P}})^{\nicefrac{1}{p}} \sum_{i\in I_d} {}_i h ^{-1} \left(\int_{\hat{P}} \abs{\hat{w}_i}^p \dd{\vec{x}}\right)^{\nicefrac{1}{p}} = (\det J_{\tilde{P}})^{\nicefrac{1}{p}} \sum_{i\in I_d} {}_i h ^{-1} \norm{\hat{w}_i}_{0,p,\hat{P}},
	\end{align*}
	\begin{equation*}
		(\det J_{\tilde{P}})^{\nicefrac{1}{p}}\norm{\hat{v}_i}_{1,p,\hat{P}} = {}_i h \sum_{\abs{\vec{\alpha}}\leq 1} h^{\vec{\alpha}} \norm{D^{\vec{\alpha}} \tilde{v}_i}_{0,p,\tilde{P}},
	\end{equation*}
	where ${}_i h = \prod_{j\in I_3\setminus\{i\}} h_j$, we compute
	\begingroup
	\allowdisplaybreaks
	\begin{align*}
		\norm{\tBDMIk{k}\tilde{\vec{v}}}_{0,p,\tilde{P}} &\leq (\det J_{\tilde{P}})^{\nicefrac{1}{p}} \sum_{i\in I_3} {}_i h^{-1} \norm{(\hBDMIk{k}\hat{\vec{v}})_i}_{0,p,\hat{P}} \\
		&\lesssim (\det J_{\tilde{P}})^{\nicefrac{1}{p}} \left[\sum_{i\in I_3} {}_i h^{-1} \left(\norm{\hat{v}_i}_{1,p,\hat{P}} + \norm{\hdiv\hat{\vec{v}}}_{0,p,\hat{P}}\right) + \frac{h_1+h_2}{h_1 h_2 h_3} \norm{\pdv{\hat{v}_3}{\hat{x}_3}}_{0,p,\hat{P}}\right]\\
		&= \sum_{i\in I_3} \left(\sum_{\abs{\vec{\alpha}}\leq 1} h^{\vec{\alpha}} \norm{D^{\vec{\alpha}}\tilde{v}_i}_{0,p,\tilde{P}} + h_i\norm{\tdiv \tilde{\vec{v}}}_{0,p,\tilde{P}}\right) + (h_1+h_2)\norm{\pdv{\tilde{v}_3}{\tilde{x}_3}}_{0,p,\tilde{P}}\\
		&\lesssim \sum_{\abs{\vec{\alpha}}\leq 1} h^{\vec{\alpha}} \norm{D^{\vec{\alpha}}\tilde{\vec{v}}}_{0,p,\tilde{P}} + h_{\tilde{P}} \norm{\tdiv \tilde{\vec{v}}}_{0,p,\tilde{P}} + (h_1+h_2)\norm{\pdv{\tilde{v}_3}{\tilde{x}_3}}_{0,p,\tilde{P}}. \qedhere
	\end{align*}
	\endgroup
\end{proof}
The estimate from the previous lemma can be brought to the general prism $P$ where the transformation is assumed to be reasonable in a certain sense.
The stability estimate on $P$ can be used to get the interpolation error estimate by a Bramble--Hilbert type argument.
The proofs for the stability and interpolation error estimates on the element $P$ follow the same steps as their analogs on simplices, see \cite[Theorems 3.5, 4.3]{ApelKempf2020} and cf. \cite[Theorems 3.1, 6.2]{AcostaApelDuranLombardi2011}, which is why they are omitted for brevity.
\begin{theorem}\label{th:BDM_prism_stability}
	Let $P$ be a prism element that emerges by the affine transformation $\vec{x} = J_P\tilde{\vec{x}}$, with  $\norm{J_P}_{\infty},\norm{J_P^{-1}}_{\infty} \leq C$, of the element $\tilde{P}\in\mathcal{R}_P$. Then for $\vec{v}\in\vec{W}^{1,p}(P)$, $1\leq p\leq\infty$, the estimate
	\begin{equation*}
		\norm{\BDMIk{k} \vec{v}}_{0,p,P} \lesssim \norm{\vec{v}}_{0,p,P} + \sum_{j\in I_3} h_j\norm{\pdv{\vec{v}}{\vec{l}_j}}_{0,p,P} + h_P \norm{\div\vec{v}}_{0,p,P} + (h_1+h_2)\norm{\pdv{v_3}{\vec{l}_3}}_{0,p,P}
	\end{equation*}
	is satisfied. The vectors $\vec{l}_j$ are the outgoing unit vectors along the edges adjacent to the transformed vertex $\vec{p}_i$.
\end{theorem}
\begin{theorem}
	Let a prism $P$ satisfy the same condition as in \Cref{th:BDM_prism_stability}. Then for $k\geq 1$, $0\leq m\leq k$ and $\vec{v}\in \vec{W}^{m+1,p}(P)$, $1\leq p\leq\infty$, the estimate
	\begin{align*}
		\norm{\vec{v}-\BDMIk{k}\vec{v}}_{0,p,P} &\lesssim \sum_{\abs{\vec{\alpha}}=m+1} h^{\vec{\alpha}} \norm{D^{\vec{\alpha}}_{\vec{l}}\vec{v}}_{0,p,P} + h_P \sum_{\abs{\vec{\alpha}}=m} h^{\vec{\alpha}} \norm{D^{\vec{\alpha}}_{\vec{l}} \div\vec{v}}_{0,p,P}\\
		&\hphantom{\lesssim\ }+(h_1+h_2) \sum_{\abs{\vec{\alpha}}= m}h^{\vec{\alpha}} \norm{\frac{\partial^{m+1}v_3}{\partial\vec{l}_1^{\alpha_1}\partial\vec{l}_2^{\alpha_2}\partial\vec{l}_3^{\alpha_3 +1}}}_{0,p,P}
	\end{align*}
	holds and the constant only depends on $k$.
\end{theorem}

\printbibliography

\end{document}